\documentclass[11pt]{amsart}
\usepackage{times,cite}
\allowdisplaybreaks[4]
\usepackage{amssymb}
\usepackage{mathrsfs}
\numberwithin{equation}{section}
\allowdisplaybreaks[4]
\newtheorem{Theorem}{Theorem}[section]

\newtheorem{Corollary}[Theorem]{Corollary}
\newtheorem{Lemma}[Theorem]{Lemma}
\newtheorem{Example}[Theorem]{Example}

\newtheorem{Definition}[Theorem]{Definition}

\unless\ifcsname proof\endcsname

\fi
\def\argmin{\mathop{\mathrm{arg\ min}}}
\def\supp{{\rm supp}\ }
\def\bbZ{\mathbb{Z}}

\def\bbR{\mathbb{R}}

\def\LTR{L^2(\bbR^n)}
\begin{document}

\title[Calder\'{o}n-Zygmund Operators]{Calder\'{o}n-Zygmund Operators with Non-diagonal Singularity}

\author{Kangwei Li}
\address{School of Mathematical Sciences and LPMC, Nankai University, Tianjin 300071, China}
\email{likangwei9@mail.nankai.edu.cn}

\author{Wenchang Sun}
\address{School of Mathematical Sciences and LPMC, Nankai University, Tianjin 300071, China}
\email{sunwch@nankai.edu.cn}

\thanks{This work was supported partially by the
National Natural Science Foundation of China(10971105 and 10990012).}

\begin{abstract}
In this paper, we introduce a  class of singular integral operators
which generalize Calder\'on-Zygmund operators to the more general
case, where the set of singular points of the kernel need not to be the diagonal,
but instead, it can be a general hyper curve.
We show that such operators have similar properties as ordinary
Calder\'on-Zygmund operators. In particular, we prove that they are of weak-type $(1, 1)$
and strong type $(p,p)$ for $1<p<\infty$.
\end{abstract}


\maketitle

\section{Introduction}
We say that $T$ is a  Calder\'{o}n-Zygmund operator
 if $T$ is a continuous linear operator maps $C_c^{\infty}(\bbR^n)$ into $\mathcal{D}'(\bbR^n)$
that extends to a bounded operator on $L^2(\bbR^n)$, and whose distribution kernel $K$ coincides with a function $K(x,y)$  defined away from the diagonal $x=y$ in
  $\bbR^n\times\bbR^n$ such that
\begin{eqnarray}
&& |K(x,y)|\le \frac{A}{|x-y|^n},\\
&& |K(x,y)-K(z,y)|+|K(y,x)-K(y,z)|\le \frac{A|x-z|^{\varepsilon}}{|x-y|^{n+\varepsilon}},
\end{eqnarray}
hold for some $A>0, \varepsilon>0$ whenever $|x-y|>2|x-z|$ and   for $f\in C_c^{\infty}(\bbR^n)$,
\[
  Tf(x)=\int_{\bbR^n}K(x,y)f(y)dy, \quad\mbox{$x\notin \supp(f)$}.
\]

Using the Calder\'on-Zygmund decomposition, one can prove that such an operator
is of weak type $(1, 1)$ and consequently strong type $(p,p)$.
Therefore, the classical Calder\'on-Zygmund theory is a powerful tool in many aspects of harmonic analysis and partial differential equations
\cite{DH,D,G1,G2,GT,LDY,LT,M,Stein,ST,T,YZ}.
And it has been widely studied in various
directions, e.g., see \cite{BD,CLX,CZ,CDL,DL,DGGLY,DGY,GLY} and references therein.
We refer to \cite{HL,HLY,KR,LSU,L1,L2,LOPTT,LYY,LDY} for some recent advances
of the Calder\'on-Zygmund theory.

Note that the singularity of the kernel $K$ lies in the diagonal $x=y$.
In this paper, we generalize the Calder\'on-Zygmund operator to the more general
case, where the set of singular points of the kernel $K$ can be general hyper curves.

We call $\Gamma$ a \emph{standard hyper curve} in $\bbR^n\times\bbR^n$
and denote it by $\Gamma\in SHC$ if
$\Gamma$ is the union of $r$ hyper curves $\Gamma_i$ in $\bbR^n\times \bbR^n$,
$1\le i\le r$, and satisfies the followings,
\begin{enumerate}
\item
  $\Gamma_i=\{(x,\gamma_i(x)):\,x\in\mathcal D_i\}$, where $\gamma_i$ is a mapping from   a closed domain $\mathcal D_i\subset \bbR^n$  to $\bbR^n$.

\item Each $\gamma_i$ is differentiable and the Jacobian $J_{\gamma_i}(x)$ of $\gamma_i$ has no zero  in the interior of $\mathcal D_i$.

\item  $\gamma_i(x)=\gamma_{i'}(x)$ has at most finitely many solutions
for $1\le i<i'\le r$.

\item Both $\gamma_i$ and $\gamma_i^{-1}$ satisfy the Lipschitz condition, i.e.,
there exists some positive constant  $c_{\gamma}>1$ such that
for any $x,x'\in \mathcal D_i$ and $y, y'\in\gamma_i(\mathcal D_i)$, $1\le i\le r$,
\begin{equation}\label{eq:Gamma}
|\gamma_i(x) - \gamma_i(x')| \le c_{\gamma}|x-x'|,\quad
|\gamma_i^{-1}(y) - \gamma_i^{-1}(y')| \le c_{\gamma}|y-y'|.
\end{equation}
\end{enumerate}

Note that $\mathcal D_i$ might be the same for different $i$. The set $SHC$ consists of
many types of hyper curves, e.g., see Examples~\ref{Ex:x1} and \ref{Ex:x2}.
For simplicity, we also use $\gamma$ to denote
the multi-valued function $\gamma(x):=\{\gamma_i(x):\,  \mathcal D_i\ni x, 1\le i\le r\}$.

For any set $E\subset \bbR^n$, let
$\gamma^{-1}(E) = \bigcup_{i=1}^r \gamma_i^{-1}(E)
  = \bigcup_{i=1}^r \{x:\, \gamma_i(x)\in E\}$.

Let $\rho(x,y)=\min\rho_i(x,y)$, where $\rho_i(x,y)$ is the distance from $(x,y)$ to $\Gamma_i$,
that is,
\[
  \rho_i(x,y)=\inf_{(x',y')\in\Gamma_i}|(x,y)-(x',y')|.
\]

Now we introduce a class of generalized Calder\'on-Zygmund operators $CZO_{\gamma}$,
for which
 the singularity of the kernel $K$ lies in a standard hyper curve $\Gamma$.
\begin{Definition}
Let $\Gamma\in SHC$ be a standard hyper curve in $\bbR^n\times\bbR^n$.
We call $T\in CZO_{\gamma}$ if $T$ is a continuous linear operator maps $C_c^{\infty}(\bbR^n)$ into $\mathcal{D}'(\bbR^n)$
that extends to a bounded operator on $L^2(\bbR^n)$, and whose distribution kernel $K$ coincides with a function $K(x,y)$  defined on $\Gamma^c$ such that
\begin{eqnarray}
&& |K(x,y)|\le \frac{A}{\rho(x,y)^n},\label{eq:k1}\\
&& |K(x,y) - K(x,y') | \le  \frac{A|y-y'|^{\delta}}{\rho(x,y)^{n+\delta}},
   \qquad |y-y'| \le \frac{1}{2}\rho(x,y), \label{eq:k2}\\
&& |K(x,y) - K(x',y) | \le  \frac{A|x-x'|^{\delta}}{\rho(x,y)^{n+\delta}},
   \qquad |x-x'| \le \frac{1}{2}\rho(x,y), \label{eq:k3}
\end{eqnarray}
hold for some $A, \delta>0$ and   for $f\in C_c^{\infty}(\bbR^n)$,
\[
  Tf(x)=\int_{\bbR^n}K(x,y)f(y)dy, \quad x\notin \gamma^{-1}(\supp f).
\]
\end{Definition}

It is easy to see that if $\Gamma$ is defined by $y=\gamma(x)$, where $\gamma$ is an
invertible transform from $\bbR^n$ to $\bbR^n$, then we can make $T$ an ordinary
Calder\'on-Zygmund operator with a simple change of
variables. However, if $\gamma$ is not invertible, e.g., $\Gamma$ is a closed hyper curve,
then the change of variables does not work (see Examples~\ref{Ex:x1} and \ref{Ex:x2}).
In other words, $CZO_{\gamma}$ is a proper generalization of ordinary
Calder\'on-Zygmund Operators.

We show that operators in $CZO_{\gamma}$ have similar properties as ordinary
Calder\'on-Zygmund operators. In Section 2, we consider the truncated kernels and operators
and give a formula for the difference betwwen two operators which share the same kernel.
In Section 3, we show that every operator in $CZO_{\gamma}$ is of weak-type $(1, 1)$
and strong type $(p,p)$ for $1<p<\infty$.

\subsection*{Notations and Definitions}

For a set $E\subset \bbR^n$,
$E^c = \bbR^n\setminus E$, $\overline{E}$ stands for the closure of $E$,
and $|E|$ is the Lebesgue measure of $E$.

\section{Truncated Kernels and Operators}

In this section, we study the truncated kernels and operators
for $T\in CZO_{\gamma}$.

The following result can be proved with the same arguments as that in the proof of
\cite[Proposition 8.1.1]{G2}.
\begin{Theorem}\label{thm:main1}
Let $T\in CZO_{\gamma}$ and $\varepsilon>0$. Set
\begin{equation}\label{eq:Tepsilon}
  T_\varepsilon f=\int_{\rho(x,y)\ge \varepsilon}K(x,y)f(y)dy=\int_{\bbR^n}K_{\varepsilon}(x,y)f(y)dy,
\end{equation}
where $K_{\varepsilon}(x,y)=K(x,y)\chi^{}_{\rho(x,y)\ge \varepsilon}$.
Assume that there exists a constant $B<\infty$ such that
\[
  \sup_{\varepsilon>0}\|T_\varepsilon\|_{L^2\rightarrow L^2}\le B.
\]
Then there exists a linear operator $T_0$ defined on $L^2(\bbR^n)$ such that
\begin{enumerate}
\item The Schwartz kernel of $T_0$ coincides with $K$ on $\Gamma^c$.
\item There exists some sequence $\varepsilon_j\downarrow 0$ such that
        \[
          \int_{\bbR^n}(T_{\varepsilon_j}f)(x)g(x)dx\rightarrow  \int_{\bbR^n}(T_0f)(x)g(x)dx
        \]
        as $j\rightarrow\infty$ for all $f,g\in L^2(\bbR^n)$.
\item $T_0$ is bounded on $L^2(\bbR^n)$ with norm $\|T_0\|_{L^2\rightarrow L^2}\le B$.
\end{enumerate}
\end{Theorem}

To consider the difference between $T$ and $T_0$, we need more information on
the set $\Gamma$.

Let $\mathcal D\subset \bbR^n$ and  $\{I_j:\,j\in J\}$  be a sequence of subsets of $\mathcal D$.
We say that  $\{I_j:\,j\in J\}$   forms a partition of $\mathcal D$
if $|I_j \cap I_{j'}|=0$ for $j\ne j'$   and $|\mathcal D\setminus \bigcup_{j\in J} I_j|=0$.

\begin{Lemma}\label{Lm:L1}
There is a sequence of closed cubes $\{I_j:\,j\in J\}$ which forms a partition
of $\bigcup_{i=1}^r \gamma_i(\mathcal D_i)$ such that
$\gamma_i^{-1}(I_j) \cap \gamma_{i'}^{-1}(I_j) =\emptyset$ for any
$j\in J$ and $1\le i<  i' \le r$.
\end{Lemma}

\begin{proof}
Let $\mathcal Y=\bigcup_{i=1}^r \gamma_i(\mathcal D_i)$ and
$Y = \{y\in\bbR^n: $  there exist $1\le i<i'\le r$  and $x\in\bbR^n
\mbox{ such that } y=\gamma_i(x)=\gamma_{i'}(x)\}$. Then $Y$ has only finitely many elements.

Since every $\gamma_i$ is continuous,
for each $y$ in the interior of $\mathcal Y\setminus Y$, there is some cube of the form
$I_y:=\prod_{k=1}^n [\frac{l_k}{2^m},\frac{l_k+1}{2^m}]$, where $m, l_k\in\bbZ$ and $m>0$,
such that $y\in I_y\subset \mathcal Y$ and
$\gamma_i^{-1}(I_y) \cap \gamma_{i'}^{-1}(I_y) =\emptyset$ for any $1\le i<  i' \le r$.

It follows that for different $I_y$'s, either they are
mutually disjointed
or one is contained in the other.
Since $\mathcal Y\setminus Y = \bigcup_{y\in \mathcal Y\setminus Y} I_y$
and  $\{I_y:\,y\in Y\}$ is at most countable,
we get the conclusion.
\end{proof}

\begin{Theorem}\label{thm:T0}
Let the hypotheses be as in Theorem~\ref{thm:main1}.
Then there exist measurable functions $b_i$ on $\mathcal D_i$  such that
\begin{equation}\label{eq:thm1:e0}
        (T-T_0)f(x)= \sum_{i=1}^r b_i(x) f(\gamma_i(x)) \chi^{}_{\mathcal D_i}(x),\qquad a.e.
\end{equation}
for all $f\in L^2(\bbR^n)$. Moreover,
$|b_i(x)|^2 \cdot |J_{\gamma_i}^{-1}(x)|  \chi^{}_{\mathcal D_i}(x)\in L^{\infty}$,
$1\le i\le r$.
\end{Theorem}

\begin{proof}
We use notations in Theorem~\ref{thm:main1}.
By Lemma~\ref{Lm:L1}, there is a sequence of closed cubes $\{I_j:\,j\in J\}$ which forms a partition
of $\bigcup_{i=1}^r \gamma_i(\mathcal D_i)$ and satisfies
$\gamma_i^{-1}(I_j) \cap \gamma_{i'}^{-1}(I_j) =\emptyset$ for any
$j\in J$ and $1\le i<  i' \le r$.

First, we prove that for $j\in J$,
\begin{equation}\label{eq:thm1:e1}
(T_0-T)(gf)(x) = (T_0-T)(g)(x) \cdot
     f(\tilde\gamma_j(x))\chi^{}_{\tilde\gamma_j^{-1}(I_j)}(x),
\end{equation}
where $f\in \LTR$, $g$ is bounded,  $\supp f, \supp g\subset I_j$,
and $\tilde\gamma_j$ is a mapping from $\gamma^{-1}(I_j)$ to $I_j$ which is induced by $\gamma$,
i.e., for $x\in \gamma_i^{-1}(I_j)$, $\tilde\gamma_j(x) = \gamma_i(x)$.

Fix some open cube $Q\subset I_j$.
Observe that $\tilde\gamma_j^{-1}(Q) = \gamma^{-1}(Q)$
and $\overline{\gamma^{-1}(Q)} = \gamma^{-1}(\overline{Q})$.
If $x\not\in \overline{\tilde\gamma_j^{-1}(Q)}$, then
$\varepsilon^{}_Q := \min_{y\in\overline{Q}} \rho(x,y)>0$.
Consequently, for $0<\varepsilon< \varepsilon^{}_Q$,
we have
\[
  (T_{\varepsilon}-T)(g\chi^{}_Q)(x) = 0 =(T_{\varepsilon}-T)(g)(x) \cdot \chi^{}_{\tilde\gamma_j^{-1}(Q)}(x).
\]
On the other hand, if $x\in \tilde\gamma_j^{-1}(Q)$, then
$x\not\in \tilde\gamma_j^{-1}(I_j\setminus Q)
= \gamma^{-1}(I_j\setminus Q)$. Therefore,
$\min_{y\in I_j\setminus Q} \rho(x,y)>0$.
It follows that for $\varepsilon$ sufficiently small, we have
\[
  (T_{\varepsilon}-T)(g\chi^{}_{I_j\setminus Q})(x) = 0 =(T_{\varepsilon}-T)(g)(x) \cdot \chi^{}_{\tilde\gamma_j^{-1}(I_j\setminus Q)}(x).
\]
Consequently, for $x\in \tilde\gamma_j^{-1}(Q)$, we have
\begin{eqnarray*}
(T_{\varepsilon}-T)(g\chi^{}_Q)(x)
&=& (T_{\varepsilon}-T)(g)(x) - (T_{\varepsilon}-T)(g\chi^{}_{I_j\setminus Q})(x)  \\
&=& (T_{\varepsilon}-T)(g)(x) \cdot \chi^{}_{\bbR^n\setminus\tilde\gamma_j^{-1}(I_j\setminus Q)}(x)   \\
&=& (T_{\varepsilon}-T)(g)(x) \cdot \chi^{}_{ \tilde\gamma_j^{-1}(  Q)}(x).
\end{eqnarray*}
Summing up the above arguments, we get that for almost every $x$, whenever $\varepsilon$ is
sufficiently small,
\begin{eqnarray}
\!\!\!\!(T_{\varepsilon}-T)(g\chi^{}_Q)(x)
 &=& (T_{\varepsilon}-T)(g)(x) \cdot \chi^{}_{ \tilde\gamma_j^{-1}(  Q)}(x)\nonumber  \\
 &=& (T_{\varepsilon}-T)(g)(x) \cdot \chi^{}_{Q}( \tilde\gamma_j(x))
 \chi^{}_{\tilde\gamma_j^{-1}(I_j)}(x).
 \nonumber
\end{eqnarray}
Taking weak limits in the above equations, we get
\begin{equation}\label{eq:thm1:e2}
(T_0-T)(g\chi^{}_Q)(x)
 = (T_0-T)(g)(x) \cdot \chi^{}_{Q}( \tilde\gamma_j(x))\chi^{}_{\tilde\gamma_j^{-1}(I_j)}(x),
\qquad a.e.
\end{equation}
By linearity, we extend (\ref{eq:thm1:e2}) to simple functions, and then to arbitrary $f\in\LTR$
which is supported in $I_j$,
i.e.,
\begin{equation}\label{eq:thm1:e3}
(T_0-T)(g f )(x)
 = (T_0-T)(g)(x) \cdot f( \tilde\gamma_j(x))\chi^{}_{\tilde\gamma_j^{-1}(I_j)}(x),
\qquad a.e.
\end{equation}

Assume that $I_j = \prod_{k=1}^n [a_{j,k},b_{j,k}]$.
For $t_k, t'_k\in [a_{j,k}, b_{j,k}]$ with $t_k<t'_k$, define
$I_{j,t} = \prod_{k=1}^n [a_{j,k},t_k]$ and
$I_{j,t'} = \prod_{k=1}^n [a_{j,k},t'_k]$.
We have
\begin{eqnarray}
(T_0-T)(\chi^{}_{I_{j,t}})(x)
&=&
   (T_0-T)(\chi^{}_{I_{j,t}}\cdot \chi^{}_{I_{j,t'}})(x) \nonumber \\
&=&
    (T_0-T)(\chi^{}_{I_{j,t'}})(x)
    \chi^{}_{I_{j,t}}(\tilde\gamma_j(x))\chi^{}_{\tilde\gamma_j^{-1}(I_j)}(x). \nonumber
\end{eqnarray}
It follows that for $x\in  \tilde\gamma_j^{-1}(I_{j,t})$, we have
\[
  (T_0-T)(\chi^{}_{I_{j,t}})(x) = (T_0-T)(\chi^{}_{I_{j,t'}})(x).
\]
Consequently, there is a function $h_j$ defined on $\tilde\gamma_j^{-1}(I_{j})$ such that
\[
  h_j(x) =  (T_0-T)(\chi^{}_{I_{j,t}})(x),\qquad x\in  \tilde\gamma_j^{-1}(I_{j,t}).
\]

For any $f\in\LTR$, let $f_j = f\cdot \chi^{}_{I_j}$. Then we have
\begin{eqnarray*}
(T_0 - T)(f_j\cdot \chi^{}_{I_{j,t}})(x)
&=&  f_j( \tilde\gamma_j(x)) (T_0 - T)( \chi^{}_{I_{j,t}})(x)
     \chi^{}_{\tilde\gamma_j^{-1}(I_j)}(x)\\
&=&       f_j( \tilde\gamma_j(x))
     h_j(x),\qquad x\in  \tilde\gamma_j^{-1}(I_{j,t}).
\end{eqnarray*}
By letting $(t_1,\cdots,t_n) \rightarrow (b_{j,1},\cdots,b_{j,n})$, we get
\[
(T_0 - T)(f_j)(x)
=       f_j( \tilde\gamma_j(x))
     h_j(x),\qquad x\in  \tilde\gamma_j^{-1}(I_{j}).
\]
Observe that $(T_0 - T)(f_j)(x)=0$ for $x\not\in \tilde\gamma_j^{-1}(I_{j})$.
We have
\begin{equation}\label{eq:T0T}
(T_0 - T)(f_j)(x)
=       f_j( \tilde\gamma_j(x))
     h_j(x)\chi^{}_{\tilde\gamma_j^{-1}(I_{j})}(x),\qquad a.e.
\end{equation}
Since $(T_0 - T)(f)(x) = 0$ whenever $\supp f\subset (\bigcup_{i=1}^r \gamma_i(\mathcal D_i))^c$,
we have
\begin{eqnarray*}
(T_0 -  T)(f)(x) =   \sum_{j\in J}  f_j( \tilde\gamma_j(x))
     h_j(x)\chi^{}_{\tilde\gamma_j^{-1}(I_{j})}(x),\qquad a.e.
\end{eqnarray*}
Note that for
every $j\in J$, we have  $\tilde\gamma_j^{-1}(I_{j}) = \bigcup_{i=1}^r \gamma_i^{-1}(I_j)$.
Hence
\begin{eqnarray*}
(T_0 -  T)(f)(x)
&=&   \sum_{j\in J}  f_j( \tilde\gamma_j(x))
     h_j(x)  \sum_{i=1}^r \chi^{}_{\gamma_i^{-1}(I_j)}(x) \\
&=&  \sum_{i=1}^r
    \sum_{j\in J}  f_j( \tilde\gamma_j(x))  h_j(x)  \chi^{}_{\gamma_i^{-1}(I_j)}(x)\\
&=&  \sum_{i=1}^r    f( \gamma_i(x))
    \sum_{j\in J}  h_j(x)  \chi^{}_{\gamma_i^{-1}(I_j)}(x).
\end{eqnarray*}
By setting $b_i(x) = \sum_{j\in J}  h_j(x)  \chi^{}_{\gamma_i^{-1}(I_j)}(x)$, we get
(\ref{eq:thm1:e0}).

Take  some $f\in\LTR$ and $j\in J$. By (\ref{eq:T0T}), we have
\begin{eqnarray*}
\|(T_0-T)(f\cdot\chi^{}_{I_j})\|_2^2
&=& \left\|
    \sum_{i=1}^r f(\gamma_i(x)) b_i(x) \chi^{}_{\gamma_i^{-1}(I_j)}(x)
   \right\|_2^2 \\
&=&
    \sum_{i=1}^r \int_{\gamma_i^{-1}(I_j)} \left|f(\gamma_i(x)) b_i(x) \right|^2 dx \\
&=&
    \sum_{i=1}^r \int_{I_j} \left|f(y) b_i(\gamma_i^{-1}(y)) \right|^2 |J_{\gamma_i^{-1}}(y)|
     dx.
\end{eqnarray*}
Since $\|(T_0-T)(f\cdot\chi^{}_{I_j})\|_2^2 \le (\|T\|_{L^2\rightarrow L^2} + B)^2
\|f\cdot\chi^{}_{I_j}\|_2^2$ for any $f\in\LTR$, we have
\[
  \sum_{i=1}^r \left|  b_i(\gamma_i^{-1}(y)) \right|^2 |J_{\gamma_i^{-1}}(y)|
   \le (\|T\|_{L^2\rightarrow L^2} + B)^2, \qquad a.e. \mbox{ on } I_j.
\]
Hence
\[
  \sum_{i=1}^r \left|  b_i(x) \right|^2 |J_{\gamma_i}^{-1}(x)|
   \le (\|T\|_{L^2\rightarrow L^2} + B)^2, \qquad a.e. \mbox{ on } \gamma_i^{-1}(I_j).
\]
Therefore, $\left|  b_i(x) \right|^2 |J_{\gamma_i}^{-1}(x)|
  \cdot \chi^{}_{\mathcal D_i}(x)\in L^{\infty}$, $1\le i\le r$.
This completes the proof.
\end{proof}

The following is an immediate consequence, which gives the difference between two
operators which share the same kernel.

\begin{Corollary}\label{Co:c1}
Let $S$ and $T$ be two operators in $CZO_{\gamma}$ which share the same kernel $K$.
Then there exist measurable functions $b_i$ on $\mathcal D_i$  such that
\begin{equation}\label{eq:c1:e0}
        (S-T)f(x)= \sum_{i=1}^r b_i(x) f(\gamma_i(x)) \chi^{}_{\mathcal D_i}(x),\qquad a.e.
\end{equation}
for all $f\in L^2(\bbR^n)$. Moreover,
$|b_i(x)|^2 \cdot |J_{\gamma_i}^{-1}(x)|  \chi^{}_{\mathcal D_i}(x)\in L^{\infty}$,
$1\le i\le r$.
\end{Corollary}

\begin{proof}
Let $(S-T)_{\varepsilon}$ be defined similarly as $T_{\varepsilon}$ in (\ref{eq:Tepsilon}).
Since $S-T$ has the kernel zero,
we have $(S-T)_0:= \lim_{\varepsilon\rightarrow 0} (S - T)_{\varepsilon}=0$.
Consequently,
\[
  S-T = S-T - (S-T)_0.
\]
Now the conclusion follows by Theorem~\ref{thm:T0}.
\end{proof}

Now we give two examples.
In the first example, $\Gamma$ consists of two hyper curves which have one common point.
\begin{Example}\label{Ex:x1}
let $\gamma(x)= \pm x$. It is easy to check that
$\gamma$ determines a standard hyper curve in $\bbR^n\times \bbR^n$.
By Theorem~\ref{thm:T0},
for any $T\in CZO_{\gamma}$, we can find $b_1, b_2\in L^{\infty}$ such that
\[
  (T-T_0)(f)(x) = b_1 (x)f(x) + b_2(x) f(-x), \qquad x\in\bbR^n,
\]
\end{Example}

And in the next example, $\Gamma$ is a diamond.
\begin{Example}\label{Ex:x2}
Suppose that $n=1$.
Let $\gamma(x) = \pm(1-|x|)$ for $|x|\le1$, and $0$ for $|x|\ge 1$.
Then $\gamma$ determines a standard hyper curve in $\bbR\times \bbR$.
By Theorem~\ref{thm:T0},
for any $T\in CZO_{\gamma}$,  we can find $b_1, b_2\in L^{\infty}$ such that
\[
  (T-T_0)(f)(x) =  \begin{cases}
  b_1 (x)f(1-|x|) + b_2(x) f(|x|-1) , & |x|\le 1,
          \\
          0, &|x|>1.
  \end{cases}
\]
\end{Example}

\section{Weak Type (1,1) and the $L^p$ boundedness}

It is well known that a Calder\'on-Zygmund operator is of weak type $(1,1)$ and strong type $(p,p)$
for $1 <p<\infty$.
In this section, we show that operators in $ CZO_{\gamma}$
have the same property.

First, we give some properties of $\rho(x,y)$.

For any $x,y\in\bbR^n$, let $\xi_{i,x}$ be the point in $\mathcal D_i$ which is most
close to $x$, i.e.,
\[
  \xi_{i,x} = \argmin_{x'\in\mathcal D_i} |x-x'|.
\]
Similarly,
\[
  \eta_{i,y} = \argmin_{y'\in \gamma_i(\mathcal D_i)} |y-y'|.
\]
For a cube $Q$, we define $\eta_{i,Q} = \{\eta_{i,y}:\,y\in Q\}$.

For any $x,y\in\bbR^n$, let
\begin{eqnarray}
\tilde\rho_i(x,y) &=& |x-\xi_{i,x}| + |y - \gamma_i(\xi_{i,x})|, \label{eq:tilde:rho}
          \\
\tilde\rho^*_i(x,y) &=& |y-\eta_{i,y}| + |x - \gamma_i^{-1}(\eta_{i,y})|,
\label{eq:tilde:rho:*}
\end{eqnarray}
$\tilde\rho(x,y) = \min_{1\le i\le r} \tilde\rho_i(x,y)$
and  $\tilde\rho^*(x,y) = \min_{1\le i\le r} \tilde\rho^*_i(x,y)$.

Next we show that
both
$\tilde\rho(x,y)$ and $\tilde\rho^*(x,y)$ are equivalent to $\rho(x,y)$.
\begin{Lemma}\label{Lm:rho}
For any $x,y\in\bbR^n$,  we have
\begin{eqnarray}
\rho_i(x,y) &\le& \tilde\rho_i(x,y) \le 2( c_{\gamma}+1) \rho_i(x,y),   \label{eq:rho:i1}\\
\rho_i(x,y) &\le& \tilde\rho^*_i(x,y) \le 2( c_{\gamma}+1) \rho_i(x,y). \label{eq:rho:i2}
\end{eqnarray}
\end{Lemma}

\begin{proof}
Fix some $(x,y)$. There exists some $x_0\in\mathcal D_i$ such that
$\rho_i(x,y) = |(x,y)-(x_0,\gamma_i(x_0))|$.

If $x\in \mathcal D_i$, then $\xi_{i,x} = x$.
It follows that
\begin{eqnarray*}
  \rho_i(x,y) &\le& |y-\gamma_i(x)| \\
  &\le& |y-\gamma_i(x_0)| + |\gamma_i(x_0) - \gamma_i(x)| \\
  &\le& |y-\gamma_i(x_0)| + c_{\gamma} |x_0 - x| \\
  &\le& 2^{1/2}c_{\gamma}  \rho_i(x,y).
\end{eqnarray*}
If $x\not\in \mathcal D_i$, then we have
\begin{eqnarray*}
  \rho_i(x,y)
   &\le& |x-\xi_{i,x}| + |y - \gamma_i(\xi_{i,x})| \\
   &\le& |x-x_0| + |y - \gamma_i(x_0)| + |\gamma_i(x_0) - \gamma_i(\xi_{i,x})| \\
   &\le& 2^{1/2}\rho_i(x,y) + c_{\gamma}|x_0 - \xi_{i,x}| \\
   &\le& 2^{1/2}\rho_i(x,y)+ c_{\gamma}(|x_0 - x| + |x-\xi_{i,x}|) \\
   &\le& 2^{1/2}\rho_i(x,y)+ 2 c_{\gamma}|x_0 - x|   \\
   &\le& (2c_{\gamma}+2)\rho_i(x,y).
\end{eqnarray*}
This proves (\ref{eq:rho:i1}). And (\ref{eq:rho:i2}) can be proved similarly.
\end{proof}

We see from Lemma~\ref{Lm:rho} that
\begin{eqnarray}
\rho(x,y) &\le& \tilde\rho(x,y) \le 2( c_{\gamma}+1) \rho(x,y),   \label{eq:rho:1}\\
\rho(x,y) &\le& \tilde\rho^*(x,y) \le 2( c_{\gamma}+1) \rho(x,y). \label{eq:rho:2}
\end{eqnarray}
In other words, all of
$\rho(x,y)$, $\tilde\rho(x,y)$ and $\tilde\rho^*(x,y)$ are equivalent.

With the result above, we can prove that (\ref{eq:k2}) and (\ref{eq:k3})
implies the H\"ormander condition.
\begin{Lemma}[The H\"ormander condition]\label{Lm:Hoermander}
Suppose that the kernel $K(x,y)$ satisfies (\ref{eq:k2}) and (\ref{eq:k3}).
Then there is some constant $C$ such that
\begin{eqnarray}
&& \int_{\rho(x,y)\ge 2|y-z|}|K(x,y)-K(x,z)|dx\le C,\label{eq:H1}\\
&& \int_{\rho(y,x)\ge 2|y-z|}|K(y,x)-K(z,x)|dx\le C.\label{eq:H2}
\end{eqnarray}
\end{Lemma}

\begin{proof}
We only need to prove (\ref{eq:H1}). And (\ref{eq:H2}) can be proved similarly.

Set $a=|y-z|$.
By Lemma~\ref{Lm:rho}, we have
\begin{eqnarray*}
&& \int_{\rho(x,y)\ge 2|y-z|}|K(x,y)-K(x,z)|dx \\
&\le& \int_{\rho(x,y)\ge 2a} \frac{Aa^{\delta}}{\rho(x,y)^{n+\delta}} dx \\
&\le& \int_{\tilde\rho^*(x,y)\ge 2a} \frac{A(2c_{\gamma}+2)^{n+\delta}a^{\delta}}
   {\tilde\rho^*(x,y)^{n+\delta}} dx \\
&\le& \sum_{i=1}^r
  \int_{\tilde\rho^*_i(x,y)=\tilde\rho^*(x,y)\ge 2a}
     \frac{C' a^{\delta}}
   {\tilde\rho^*_i(x,y)^{n+\delta}} dx \\
&\le& \sum_{i=1}^r\Bigg(
  \int_{\tilde\rho^*_i(x,y)\ge 2a, |x-\gamma_i^{-1}(\eta_{i,y})|\le a} \frac{C' a^{\delta}}
   {\tilde\rho^*_i(x,y)^{n+\delta}} dx   +
  \int_{\tilde\rho^*_i(x,y)\ge 2a, |x-\gamma_i^{-1}(\eta_{i,y})|\ge a} \frac{C' a^{\delta}}
   {\tilde\rho^*_i(x,y)^{n+\delta}} dx \Bigg)\\
&\le& \sum_{i=1}^r\Bigg(
  C''   +
  \int_{ |x-\gamma_i^{-1}(\eta_{i,y})|\ge a} \frac{C' a^{\delta}}
   {   |x - \gamma_i^{-1}(\eta_{i,y})|^{n+\delta}} dx \Bigg)\\
&\le& C.
\end{eqnarray*}
This completes the proof.
\end{proof}

Given a positive number $\theta$ and a cube $Q\subset\bbR^n$, we define
\[
  Q_{i,\theta} = \begin{cases}
  \{x:\, d(x,\gamma_i^{-1}(\eta_{i,Q}))\le \theta\cdot \ell(Q)\},
     & d(Q, \gamma_i(\mathcal D_i)) < 2n^{1/2}\ell(Q),\\
 \emptyset, & \mathrm{otherwise},
    \end{cases}
\]
where $\ell(Q)$ is the side length of $Q$.
Let $Q_{\theta} = \bigcup_{1\le i\le r}  Q_{i,\theta}$.


The following result is needed in the proof of weak type $(1,1)$.

\begin{Lemma} \label{Lm:Qtheta}
Let $\Gamma\in SHC$.
\begin{enumerate}
\item There is some constant $C$ such that for any   $\theta>1$ and cube $Q\subset\bbR^n$,
$|Q_{\theta}| \le C\theta^n |Q|$.

\item For $\theta>2n^{1/2}+5 n^{1/2} c_{\gamma} $ and every cube $Q$,
we have $\rho(x,y) \ge 2n^{1/2}\ell(Q)$
for all $x\not\in Q_{\theta}$ and $y\in Q$.
\end{enumerate}
\end{Lemma}

\begin{proof}
First, we prove (i).
For each $1\le i\le r$,
we only need to consider the case $d(Q, \gamma_i(\mathcal D_i)) < 2n^{1/2}\ell(Q)$.
In this case, there is some $y_i\in Q$ such that
\[
  d(y_i, \gamma_i(\mathcal D_i)) = d(Q, \gamma_i(\mathcal D_i)) < 2n^{1/2}\ell(Q).
\]
For any $x\in Q_{i,\theta}$, there is some $y_x\in Q$ such that
\[
  |x - \gamma_i^{-1}(\eta_{i,y_x})| \le \theta\cdot \ell(Q).
\]
It follows that
\begin{eqnarray*}
     |x - \gamma_i^{-1}(\eta_{i,y_i})|
&\le&  |x - \gamma_i^{-1}(\eta_{i,y_x})|
    + |\gamma_i^{-1}(\eta_{i,y_x}) - \gamma_i^{-1}(\eta_{i,y_i})|\\
&\le&   \theta\cdot \ell(Q)
    + c_{\gamma}| \eta_{i,y_i}-\eta_{i,y_x}  |\\
&\le&   \theta\cdot \ell(Q)
    + c_{\gamma}\Big(| \eta_{i,y_i} - y_i| + |y_i-y_x| + |y_x - \eta_{i,y_x} |\Big)\\
&=&   \theta\cdot \ell(Q)
    + c_{\gamma}\Big(d(y_i, \gamma_i(\mathcal D_i)) + |y_i-y_x|
     +   d(y_x, \gamma_i(\mathcal D_i))\Big)\\
&\le&   \theta\cdot \ell(Q)
    + 2c_{\gamma}\Big(d(y_i, \gamma_i(\mathcal D_i)) + |y_i-y_x|\Big)
    \\
&\le&  \Big( \theta + 6n^{1/2}c_{\gamma}\Big)\ell(Q).
\end{eqnarray*}
Hence
\[
  Q_{\theta} \subset \bigcup_{i} \left\{x:\,  |x - \gamma_i^{-1}(\eta_{i,y_i})|
    \le \left(\theta + 6 n^{1/2} c_{\gamma} \right) \ell(Q) \right\}.
\]
Since $\theta>1$, we have   $| Q_{\theta}| \le  C \theta^n |Q|$.

Next, we prove (ii).
Fix some $1\le i\le r$. There are two cases.

Case 1. $d(Q, \gamma_i(\mathcal D_i)) \ge  2n^{1/2}\ell(Q)$.
In this case,
for any $x\in\bbR^n$ and $y\in Q$,
\[
  \rho_i(x,y) = \min_{z\in \gamma_i(\mathcal D_i)}
     |(x,y) - (\gamma_i^{-1}(z), z)|
  \ge  \min_{z\in \gamma_i(\mathcal D_i)}|y-z|
  \ge  2n^{1/2}\ell(Q).
\]

Case 2. $d(Q, \gamma_i(\mathcal D_i)) <  2n^{1/2}\ell(Q)$.
We conclude that
for any $x\not\in Q_{\theta}$ and
$z\in \gamma_i(\mathcal D_i)$,
 $\max\{|x-\gamma_i^{-1}(z)|, |y-z|\}\ge 2 n^{1/2}  \ell(Q)$.

To see this, take some $y_i\in Q$ such
that $ d(y_i,\gamma_i(\mathcal D_i))$ $= d(Q, \gamma_i(\mathcal D_i)) <  2n^{1/2}\ell(Q)$.
If $|y-z| \le 2 n^{1/2}\ell(Q)$, then we have
\begin{eqnarray*}
  |x-\gamma_i^{-1}(z)|
&\ge& | x- \gamma_i^{-1}(\eta_{i,y})| - |\gamma_i^{-1}(\eta_{i,y}) - \gamma_i^{-1}(z)|
     \\
&\ge& \theta\cdot\ell(Q) - c_{\gamma}| \eta_{i,y} -z|
     \\
&\ge& \theta\cdot\ell(Q) - c_{\gamma} (|\eta_{i,y} - y| +  |y-z|)
     \\
&\ge& \theta\cdot\ell(Q) - c_{\gamma} \left(|y-y_i| + d(y_i,\gamma_i(\mathcal D_i))
  +  |y-z|\right)
     \\
&\ge& (\theta - 5 n^{1/2} c_{\gamma}) \ell(Q)\\
&\ge&  2 n^{1/2}  \ell(Q).
\end{eqnarray*}
Hence $\max\{|x-\gamma_i^{-1}(z)|, |y-z|\}\ge 2 n^{1/2}  \ell(Q)$. Therefore,
\[
  \rho_i(x,y) = \min_{z\in \gamma_i(\mathcal D_i)}
     |(x,y) - (\gamma_i^{-1}(z), z)|
  \ge  2n^{1/2}\ell(Q).
\]
In both cases, we get
$  \rho_i(x,y)  \ge  2n^{1/2}\ell(Q)$. It follows that
$\rho(x,y) \ge 2n^{1/2}\ell(Q)$
for all $x\not\in Q_{\theta}$ and $y\in Q$.
This completes the proof.
\end{proof}

The weak type $(1,1)$ and the $L^p$ boundedness for $T\in CZO_{\gamma}$
can be proved similarly to the one for ordinary Calder\'on-Zygmund operators.
To keep the paper more readable, we include a proof here.

\begin{Theorem}\label{thm:main2}
Let $T\in CZO_{\gamma}$ and $1<p<\infty$. Then there is some constant $C>0$ such that
\begin{equation}\label{eq:e4}
\|Tf\|_{L^{1,\infty}}\le C\|f\|_{L^1}
\end{equation}
and
\begin{equation}\label{eq:e5}
\|Tf\|_{L^p}\le C\|f\|_{L^p}.
\end{equation}
\end{Theorem}

\begin{proof}
First, we show that $T$ is weakly bounded on a dense set of $L^1(\bbR^n)$.
Fix some $f\in L^1\bigcap L^2(\bbR^n)$ and $\lambda>0$.
Form the Calder\'{o}n-Zygmund decomposition
 of $f$ at height $\lambda$. We get disjoint cubes $\{Q_{k}:\, k\in \mathbb{K}\}$  such that
\[
    |f(x)| \le \lambda,\qquad x\not\in \bigcup_{k\in\mathbb K} Q_k
\]
and
\[
   \lambda \le \frac{1}{|Q_k|} \int_{Q_k} |f(x)| dx \le 2^n\lambda, \qquad k\in \mathbb{K}.
\]
Define
\[
g(x)=\left\{ \begin{array}{ll}
  f(x),  &x\not\in \bigcup_{k\in\mathbb K} Q_k, \\
  |Q_{k}|^{-1}\int_{Q_{k}}f(y)dy,  &x\in \mbox{ interior of }Q_{k},
\end{array}
\right.
\]
and
\[
b_{k}(x)=\left\{\begin{array}{ll}
  0, & x\not\in Q_{k} ,\\
  f(x)-|Q_{k}|^{-1}\int_{Q_{k}}f(y)dy, &x\in \mbox{ interior of }Q_{k}.
\end{array}\right.
\]
Since $Q_k$ are disjoint, the series $b(x) := \sum_{k\in\mathbb K} b_k(x)$ is convergent in
$L^2(\bbR^n)$.
Moreover, for any $\lambda>0$,
\begin{equation}\label{eq:lambda}
  \{x:\,|(Tf)(x)| \ge \lambda\}
  \subset  \{x:\,|(Tg)(x)| \ge \frac{\lambda}{2}\} \bigcup
   \{x:\,|(Tb)(x)| \ge \frac{\lambda}{2}\}.
\end{equation}

Take some $\theta>2n^{1/2}+5 n^{1/2} c_{\gamma}$.
We see from Lemma~\ref{Lm:Qtheta}  that for any cube $Q$, $\rho(x,y)\ge 2|y-z|$
for all $x\notin  Q_{\theta}$ and $y,z\in Q$.

For each $k\in\mathbb K$, set $Q_k^*= (Q_k)_{\theta}$. Let $B^{*}=\bigcup_{k\in\mathbb K}
Q_{k}^{*}$ and
$G^{*}=\mathbb{R}^n\setminus{B}^{*}$.
Then we have
\begin{eqnarray}
    \Big|\{x:\,|Tg(x)\geq \frac{\lambda}{2}|\}\Big|
&\le& \frac{4}{\lambda^2} \int_{\{x:\,|Tg(x)\geq \frac{\lambda}{2}|\}} |Tg(x)|^2 dx    \label{eq:tf1} \\
&\le& \frac{4\|T\|_{L^2\rightarrow L^2}^2}{\lambda^2}
      \int_{\bbR^n} |g(x)|^2 dx \nonumber \\
&\le& \frac{4\|T\|_{L^2\rightarrow L^2}^2}{\lambda^2}\left(
   \int_{G} |g(x)|^2 dx + \sum_{k\in\mathbb{K}} \int_{Q_k}\!\! 2^n\lambda |g(x)| dx
       \right) \nonumber \\
&\le& \frac{4\|T\|_{L^2\rightarrow L^2}^2}{\lambda^2}\left(
   \int_{G}\! \lambda |f(x)| dx +2^n\lambda  \sum_{k\in\mathbb{K}} \int_{Q_k}\!\!\! |f(x)| dx
       \!\right) \nonumber \\
&\le&    \frac{2^{n+2}}{\lambda}\|T\|_{L^2\rightarrow L^2}^{2}\|f\|_{L^{1}}.  \nonumber 
\end{eqnarray}
On the other hand, we have
\begin{eqnarray}
   \Big|\{x:\, |Tb(x)\geq \frac{\lambda}{2}|\}\Big| 
&\le& |B^*| + \Big|\{x\in G^*:\,|Tb(x)\geq \frac{\lambda}{2}|\}\Big| \label{eq:tf2} \\
&\le& C'\theta^n \Big| \bigcup_{k\in\mathbb K} Q_k\Big| + \frac{2}{\lambda}\int_{G^*}|Tb(x)|dx\nonumber \\
&\le&  \frac{C'\theta^n}{\lambda}\|f\|_{1} +
  \frac{2}{\lambda} \sum_{k\in\mathbb K} \int_{G^*}|Tb_k(x)|dx.\nonumber
\end{eqnarray}
Observe that
\begin{eqnarray}
 \int_{G^*}|Tb_k(x)|dx
&\le&   \int_{\bbR^n\setminus Q_k^*}|Tb_k(x)|dx \label{eq:tem2} \\
&=&   \int_{\bbR^n\setminus Q_k^*} dx
  \bigg|\int_{Q_k} K(x,y) b_k(y) dy\bigg| \nonumber \\
&=&  \int_{\bbR^n\setminus Q_k^*} dx
  \bigg|\int_{Q_k} (K(x,y)- K(x,y_k)) b_k(y) dy\bigg| \nonumber \\
&\le&   \int_{Q_k} | b_k(y)| dy
   \int_{\rho(x,y)\ge 2|y-y_k|} |K(x,y)- K(x,y_k)| dx \nonumber\\
&\le&  C\int_{Q_k} | b_k(y)| dy,\nonumber
\end{eqnarray}
where  $y_k$ is the center of $Q_k$
and we use Lemma~\ref{Lm:Hoermander} in the last step.
Now we see from (\ref{eq:tem2}) that
\begin{eqnarray}
 \sum_{k\in\mathbb K} \int_{G^{*}}|Tb_k(x)|dx
&\leq&
C
\sum_{k\in\mathbb K}\int_{Q_{k}}|b_{k}(y)|dy  \label{eq:tf3} \\
&\le&
 C\sum_{k\in\mathbb K}\int_{Q_{k}}\Big(|f(y)|+\frac{1}{|Q_{k}|}
  \int_{Q_{k}}|f(x)|dx\Big)dy \nonumber \\
&=&2C \sum_{k\in\mathbb K}\int_{Q_{k}}|f(y)|dy  \nonumber \\
&\le&2C\|f\|_{L^{1}}.\nonumber
\end{eqnarray}
Since $f(x) = g(x) + b(x)$, together with (\ref{eq:tf1}), (\ref{eq:tf2}) and  (\ref{eq:tf3}),
we have
\begin{equation}\label{eq:wkL1}
\|Tf\|_{L_{weak}^{1}}
\le \left( 2^{n+2}\|T\|_{L^{2}\rightarrow L^{2}}^{2}+C'\theta^n
+\, 4C\right) \|f\|_1,
\end{equation}
where $f\in L^1\bigcap L^2(\bbR^n)$.

For any $f\in L^1(\bbR^n)$, we can find a sequence $\{f_k:\,k\ge 1\}
\subset L^1\cap L^2(\bbR^n)$ such that
$\|f-f_k\|_1\rightarrow 0$. By (\ref{eq:wkL1}), we have
\[
    \Big|\{x:\, |(T(f_k-f_{k'}))(x)|\ge \lambda\}\Big|
      \le \frac{C''}{\lambda} \|f_k -f_{k'}\|_1.
\]
Hence $\{Tf_k:\, k\ge 1\}$ is convergent in measure.
Since $\{f_k:\,k\ge 1\}$ is arbitrary, it is easy to see that the
limit, denoted by $Tf$, is independent of
choices of  $\{f_k:\,k\ge 1\}$.
Hence, (\ref{eq:wkL1}) is true for any $f\in L^1(\bbR^n)$.
This proves $(\ref{eq:e4})$.

By the interpolation theorem, we prove $(\ref{eq:e5})$ with $1<p<2$.
And by a duality argument, we prove $(\ref{eq:e5})$ with $p>2$.
\end{proof}

We see from Theorem~\ref{thm:main2} that an operator $T\in CZO_{\gamma}$ is well defined on $L^p(\bbR)$ for $1\le p<\infty$.
As for ordinary Calder\'on-Zygmund operators, $Tf$ can also be expressed as integrals in some cases. Specifically,
we have the following result, which can be proved similarly to \cite[Proposition 8.2.2]{G2}.

\begin{Theorem}
Let $T$ be  an operator in $CZO_{\gamma}$ associated with some kernel $K$.
Then for $f\in L^p(\bbR^n)$, $1\le p<\infty$, the following absolutely convergent integral representation is valid:
\[
  (Tf)(x) = \int_{\bbR^n} K(x,y) f(y)dy, \qquad \mbox{ a.e. on } \gamma^{-1}(\supp f).
\]
\end{Theorem}

\textbf{Acknowledgements.}\,\,
The authors thank Professor Dachun Yang for
very useful discussions and suggestions.

\end{document}